\documentclass[11pt]{amsart}
\usepackage{amssymb,latexsym}
\usepackage{enumerate}
\usepackage{hyperref}
\usepackage{fullpage}
\usepackage{mathabx}

\makeatletter \@namedef{subjclassname@2010}{%
  \textup{2010} Mathematics Subject Classification}
\makeatother

\newcounter{thm} \numberwithin{thm}{section}
\newtheorem{Theorem}[thm]{Theorem}

\newtheorem{Corollary}[thm]{Corollary}
\newtheorem*{Conjecture}{Conjecture}

\newtheorem*{UnnumberedTheorem}{Theorem}

\newcommand{\CC}[0]{\mathbb C}

\newcommand{\FF}[0]{\mathbb F}

\newcommand{\eps}[0]{\varepsilon}

\renewcommand{\mod}[1]{\ (\text{mod }#1)}

\newcommand{\lr}[1]{\left(#1\right)}

\begin{document}

% %%%% To ease editing, for IMPAN journals add:

\baselineskip=17pt

% %%%%%%%%%%

% % In the running head, replace first names by initials % and give an
% abbreviation of the title.

\title{Refined Estimates Concerning Sumsets Contained in the Roots of Unity}

\author[Brandon Hanson]{Brandon Hanson} \address{Department of Mathematics, University of Georgia, Athens, GA 30602.}
\author[Giorgis Petridis]{Giorgis Petridis} \address{Department of Mathematics, University of Georgia, Athens, GA 30602.}
\date{}
\maketitle

\begin{abstract}
We prove that the clique number of the Paley graph is at most $\sqrt{p/2} + 1$, and that any supposed additive decompositions of the set of quadratic residues can only come from co-Sidon sets.
\end{abstract}

\let\thefootnote\relax\footnotetext{The authors are supported by the NSF Award 1723016 and gratefully acknowledge the support from the RTG in Algebraic Geometry, Algebra, and Number Theory at the University of Georgia, and from the NSF RTG grant DMS-1344994. The second author is supported by the NSF Award 1723016. 

Mathematics Subject Classification (2010) code: 11B30.}

\section{Introduction}
Let $Z_d$ denote the $d$'th roots of unity belonging to a field $\FF$, which is to say, the solutions to $z^d=1$. Then $Z_d$ forms a multiplicative subgroup of the units of $\FF$, and number theoretic intuition leads one to expect that $Z_d$ not possess too much additive structure. There are a number of ways to interpret such a statement, and in this article we will be concerned with sumset decompositions
\[Z_d=A+B\] for non-singleton subsets $A$ and $B$ of $\FF$, where
\[A+B=\{a+b:a\in A,\ b\in B\}.\] 
Since sumsets have some additive structure, this type of decomposition seems unlikely outside of very particular situations. On the other hand, since we impose no further constraints on $A$ and $B$, their sumset is generally not structured enough to make strong additive statement, and problems involving sumset decompositions are often very difficult to get a handle on for this reason.

One heuristic for there to be no sumset decomposition of $Z_d$ is that addition is linear, and the property of being a multiplicative subgroup feels very much non-linear. When $\FF=\CC$, this is solidified by the fact that the roots of unity are cocircular. We dispose of this case here.

\begin{Theorem}
If $d>4$ then $Z_d=A+B$ has no non-trivial solutions over $\CC$.
\end{Theorem}
\begin{proof}
Assume $2\leq |A|\leq |B|$. Let $a_1,a_2\in A$ be distinct. Then, since $Z_d\subset S^1$ (the unit circle in the complex plane), we have that $a_1+B$ and $a_2+B$ are subsets of $S^1$. But
\[a_1-a_2=(a_1+b)-(a_2+b)\]
has at least $|B|$ representations as a difference in $S^1-S^1$. On the other hand, any non-zero complex number is a difference of two points on the unit circle in at most two ways. Thus $|B|=2$ and $d=4$.
\end{proof}
We remark that the condition $d>4$ is crucial here. If $d<4$, $Z_d$ is too small to admit a non-trivial sumset decomposition, and when $d=4$
\[\left\{\frac{1+i}{2},-\frac{1+i}{2}\right\}+\left\{\frac{1-i}{2},-\frac{1-i}{2}\right\}=Z_4\] yields a non-trivial sumset decomposition of the fourth roots. 

Turning to finite fields, the full group of units is now contained in a line, so the above argument is lost. Indeed, if $p\geq 5$ then 
\[\left\{0,\frac{p-1}{2}\right\}+\left\{1,\ldots,\frac{p-1}{2}\right\}=\{1,\ldots,p-1\}=Z_{p-1}\] 
is a non-trivial sumset decomposition of the full group of units. However, we expect that no such decomposition holds when $4<d<p-1$. In this setting, sumsets contained in $Z_d$ and, more broadly, the additive distribution of $Z_d$ has a rich history. Of particular interest is the case $d=(p-1)/2$ so that $Z_d$ consists of the quadratic residues. Traditionally, attacks on these problems make use of cancellation in character sums such as
\[S_\chi(A,B)=\sum_{a\in A}\sum_{b\in B}\chi(a+b)\]
where \[\chi:\FF_p^\times\to\CC\]
is a multiplicative homomorphism of $\FF_p^\times$, often extended to $\FF_p$ via $\chi(0)=0$. 
Basic estimates for $S_\chi(A,B)$ date back at least as early as Vinogradov (e.g. Exercise 8 in Chapter 5 of \cite{Vinogradov}).
\begin{UnnumberedTheorem}[Vinogradov]
For any subsets $A$ and $B$ of $\FF_p$ and any non-trivial multiplicative character $\chi$, we have the estimate
\[|S_\chi(A,B)|\leq \sqrt{p|A||B|}.\]
\end{UnnumberedTheorem}
It is at this point worth observing that the sum $S_\chi(A,B)$ can never be larger than $|A||B|$ in modulus, as it is a sum of numbers $\chi(x)$ of modulus at most 1. Any improvement on the estimate $|A||B|$ is called non-trivial, and so we see that Vinogradov's theorem provides a non-trivial estimate when
\[|A||B|>p.\] An interesting feature of Vinogradov's estimate is that while it is totally elementary to prove, it is still the best known estimate for $S_\chi(A,B)$. However, a much stronger range of non-trivial estimates is conjectured (the conjecture is folklore and may trace back to Vinogradov).
\begin{Conjecture}
For any $\eps>0$, once $p$ is sufficiently large in terms of $\eps$, then
\[|S_\chi(A,B)|<|A||B|\]
for all sets $A,B\subseteq \FF_p$ subject only to the condition $|A|,|B|>p^\eps$.
\end{Conjecture}

In the case that $A$ or $B$ is highly structured, Vinogradov's estimate can be improved. Theorems of this sort can be found in \cite{Burgess1}, \cite{Burgess2}, \cite{FI}, \cite{Chang}, \cite{ShkredovShparlinski}, \cite{ShkredovVolostnov} and \cite{Volostnov}. For completely general sets, \cite{BMR} provides a slight improvement for certain primes, while one also has a very modest improvement if we replace a two-fold convolution with a three-fold convolution, which means estimating
\[S_\chi(A,B,C)=\sum_{a\in A}\sum_{b\in B}\sum_{c\in C}\chi(a+b+c).\]
This was carried out by the first author in \cite{Hanson}. 

A particular implication of Vinogradov's estimate is that if $A$ and $B$ are subsets of $\FF_p$ with the property that $A+B\subseteq Z_{d}$ for some $d$ properly dividing $p-1$, then we must have $|A||B|\leq p$.

The main theorem of this paper is a refinement of these estimates and relies on Stepanov's method of auxiliary polynomials. See \cite{Heath-BrownKonyagin} for another application of Stepanov's method to an additive problem of roots of unity.

\begin{Theorem}\label{Main1}
Let $p$ be a prime and suppose $A,B\subseteq\FF_p$ satisfy $A+B\subseteq Z_d \cup \{0\}$ for some $d$ properly dividing $p-1$. Then 
\[|A||B|\leq d + |B \cap (-A)|.\]
\end{Theorem}

When $d=(p-1)/2$ and $A+B \subseteq Z_\frac{p-1}{2}$, this theorem just fails to confirm the following well-known conjecture of S\'ark\"ozy, \cite{Sarkozy}.
\begin{Conjecture}[S\'ark\"ozy]
For any $p>3$, there is no non-trivial additive decomposition of the set $Z_\frac{p-1}{2}$.
\end{Conjecture}
Further reading on this conjecture can be found in \cite{Sarkozy}, \cite{Shkredov2}, \cite{Shkredov3}, and \cite{Shparlinski}. With our theorem, we have made the following contribution to S\'ark\"ozy's conjecture. 
\begin{Corollary}\label{UniqueSums}
Let $p$ be a prime and suppose $A,B\subseteq\FF_p$ satisfy $A+B=Z_d$ for some $d$ properly dividing $p-1$. Then 
\[|A||B|=d\]
and all sums $a+b$ are distinct. In particular, if $d$ is prime and neither $A$ nor $B$ is a singleton, no such decomposition is possible.
\end{Corollary}

After the first draft of this article was made available, George Shakan observed that by combining previous work on S\'ark\"ozy's conjecture and Ford's work on divisors, Theorem \ref{Main1} implies that S\'ark\"ozy's conjecture must in fact hold for almost all primes. Below, $\pi(x)$ denotes the number of primes $p\leq x$.

\begin{Corollary}[Shakan]
As $x$ tends to infinity, the number of primes $p\leq x$ such that there exist non-singleton sets $A,B\subseteq\FF_p$ with $A+B=Z_{\frac{p-1}{2}}$ is $o(\pi(x))$.
\end{Corollary}
\begin{proof}
Call the primes in question bad primes. It will suffice to show that for $x$ sufficiently large, the number of bad primes $p$ with $x/e\leq p\leq x$ is $o(x/\log x)$. Indeed, if this is the case, then there are at most
\[\sum_{k=1}^{\log x}o(e^k/k)=o\lr{\int_1^x \frac{dt}{\log t}}=o\lr{\pi(x)}\]
bad primes up to $x$, by the Prime Number Theorem.
So, let $p$ be a bad prime in the interval $[x/e,x]$ and suppose $A+B=Z_{\frac{p-1}{2}}$ for subsets $A,B\subset\FF_p$ with $|A|=M>1$ and $|B|=N>1$. By Corollary \ref{UniqueSums}, \[p-1=2MN,\] and moreover Shkredov has proved that
\[\lr{\frac{1}{6}-o(1)}\sqrt{p}\leq M,N\leq (3+o(1))\sqrt p,\]
see \cite[Corollary 2.6]{Shkredov2}. It turns out that primes of such a form are rare. To see this we proceed as follows. For $0 < y < z$, let we let $\tau(n; y, z)$ denote the number of divisors $d$ of $n$ which satisfy $y<d\leq z$, and for $x\geq 1$ we write \[H(x,y,z)=|\{n: 1\leq n\leq x,\ \tau(n;y,z)\geq 1\}|,\]
and
\[P(x,y,z)=|\{p: 1\leq p\leq x,\ p\text{ prime},\ \tau(p-1;y,z)\geq 1\}|.\]
Assuming (as we may) that $M\leq N$, then $M$ is a divisor $d$ of $p-1$ with 
\[\frac{1}{100}\sqrt x\leq d\leq \sqrt x\]
provided $p$ is sufficiently large. So with $z=\sqrt x$ and $y=\sqrt x/100$ we have \[P(x,y,z)\ll \frac{H(x,y,z)}{\log x}\] by Theorem 6 of \cite{Ford}. Next by Theorem 1 part (v) of \cite{Ford}, with $u=\log(100)/\log(y)$, we have $H(x,y,z)=o(x)$ (in a stronger quantitative sense, in fact), and the corollary is proved.
\end{proof}

A related question is the estimation of the clique number $\omega(G_p)$ of the Paley graph $G_p$. Recall that when $p=1\mod 4$, the Paley graph $G_p$ is the Cayley graph on the additive group $\FF_p$ generated by the quadratic residues. The best known bound follows from Vinogradov's estimate, and \cite{MaistrelliPenman} provides the slight improvement for certain primes $\omega(G_p)\leq \sqrt{p-4}$. In a similar vein to the above considerations, when $A+B = Z_d \cup \{0\}$ for some $d$ properly dividing $p-1$, all non-zero sums $a+b$ must be distinct.  In the case $A=-B$, the theorem implies the following bound on the clique number of Paley graphs.

\begin{Corollary}
Let $p$ be a prime, $d$ properly dividing $p-1$ and suppose $A\subseteq\FF_p$ is such that $A-A\subseteq Z_d\cup\{0\}$. Then $|A|(|A|-1)\leq d$. In particular, for $p = 1 \mod 4$, we have \[\omega(G_p)\leq \frac{\sqrt{2p-1}+1}{2}.\]
\end{Corollary}

Paley graphs are also defined over a field with $q$ elements for all prime powers $q$ that are congruent to $1$ modulo $4$. It is noted in \cite{MaistrelliPenman} that when $q$ is the square of a prime, the clique number equals $\sqrt{q}$ (this was proven in \cite{BDR}). The method developed to prove Theorem~\ref{Main1} only works in prime fields. 

A third corollary is a variant of S\'ark\"ozy's Conjecture which concerns decompositions of the form
\[Z_d=A\dotdiv A=\{a-a':a,a'\in A, a\neq a'\}.\] This type of decomposition was considered in \cite{LevSonn} with the extra property that every element of $Z_d$ has a unique representation in $A-A$. It was noted there that this is possible for only a thin set of primes. As a direct corollary from Theorem~\ref{Main1}, this property is necessitated by the very existence of the decomposition.
\begin{Corollary}
Let $p$ be a prime, $d$ be even and properly divide $p-1$ and suppose $A\subseteq\FF_p$ is such that $A\dotdiv A =  Z_d$. Then each difference $a-a'$ with $a\neq a'$ is unique. Consequently, when $d= \frac{p-1}{2}$ then the prime $p$ is of the form $\frac{n^2+1}{2}$ for some odd integer $n$, and hence there is no decomposition $A\dotdiv A =  Z_d$ for almost every prime.
\end{Corollary}
\section*{Acknowledgements}
The content and exposition in this article has benefited from insightful discussions with Seva Lev, Neil Lyall, \'Akos Magyar, David Masser, Maksym Radziwi\l\l, Tom Sanders, George Shakan, and Ilya Shkredov.

\section{Proof of Theorem 1.2}

\begin{proof}
 Suppose $|A| \leq |B|$, enumerate $A$ as $A=\{a_1,\ldots,a_M\}$, and let $r = |B \cap (-A)|$. When $M=1$ the result is immediate. When $M>1$ we make use of Stepanov's method of auxiliary polynomials. Let $c_1,\ldots,c_M$ be not all zero and subject to the constraint that the polynomial
\[G(x)=\sum_{k=1}^M c_k(x+a_k)^{M-1}=\sum_{j=0}^{M-1}x^{j}\binom{M-1}{j}\sum_{k=1}^Mc_ka_k^{M-1-j}.\]
is a constant. This is possible since the system in question is of Vandermonde type, there are $M-1$ coefficients of non-constant powers of $x$ in $G$ that must be made zero, and $M$ variables at our disposal. We may choose $c_1, \dots, c_M$ so that the resulting constant equals 1:
\[1=\sum_{k=1}^Mc_ka_k^{M-1}.\] Since $G(x)$ is a constant, $G^{(j)}(b)=0$ for any $b$ and $j>0$. 

Now let 
\[F(x)=-1+\sum_{k=1}^M c_k(x+a_k)^{D}\]
where the values of $c_k$ are as above and $D=d+M-1$. Observe that in view of the fact
\[M=|A|\leq |A+B| \leq \frac{p-1}{2}+1,\] we have \[D\leq \frac{p-1}{2}+M-1\leq p-1.\] By the Binomial Theorem, the $M$ leading coefficients of $F$ are of the form
\[\binom{D}{l}\sum_{k}c_ka_k^{l},\]
which vanish for $l<M-1$ by our choice of $c_k$ but not for $l=M-1$. Hence $F$ has degree $D-(M-1)=d$ and is, in particular, non-zero.

Next, for $b\in B \cap (-A)$ and all $1\leq k \leq M$, $(b+a_k)^{d+1}=(b+a_k)$ so
\[F(b)=-1+\sum_{k=1}^Mc_k(b+a_k)^{d+1}(b+a_k)^{M-2}= -1+\sum_{k=1}^M c_k(b+a_k)^{M-1}= -1+G(b)=0.\]
Similarly, if we differentiate $F$ a total of $j$ times, where $1\leq j\leq M-2$, we get
\[F^{(j)}(x)=(D)\cdots(D-j+1)\sum_{k=1}^Mc_k(x+a_k)^{D-j}\]
whence 
\[F^{(j)}(b)=\frac{(D)\cdots(D-j+1)}{(M-1)\cdots(M-j)}G^{(j)}(b)=0.\]
In this way $F$ has a root of order $M-1$ at each $b\in B \cap (-A)$.

If $b \in B \setminus (-A)$, $(b+a_k)^{d}=1$ for all $1\leq k \leq M$. The above argument gives that $F$ has a root of order $M$ at each $b\in B \setminus (-A)$. Therefore
\[(M-1) r + M (|B|-r)\leq \deg F=d\]
yielding the theorem.
\end{proof}

\phantomsection

\end{document}